\documentclass[12pt]{amsart}
\usepackage[colorlinks=true,linkcolor=magenta,citecolor=cyan,urlcolor=olive]{hyperref}
\usepackage{cleveref}
\usepackage{amssymb,float,graphicx,mathrsfs,paralist,tikz,upgreek,xcolor}
\usepackage[top=2.5cm,bottom=2.5cm,left=3cm,right=3cm]{geometry}
\usetikzlibrary{arrows.meta}

\newcommand{\rad}{\mathrm{rad}}
\newcommand{\fm}{\mathfrak m}
\newcommand{\bH}{\mathbb H}
\renewcommand{\d}{\mathrm d}
\newcommand{\bR}{\mathbb R}
\newcommand{\e}{\mathrm e}
\renewcommand{\Re}{\operatorname{Re}}

\newcommand{\0}{\mathbf 0}
\newcommand{\1}{\mathbf 1}
\newcommand{\bg}{\mathbf g}
\newcommand{\arcosh}{\operatorname{arch}}
\newcommand{\bS}{\mathbb S}
\newcommand{\fe}{\mathfrak e}
\renewcommand{\cosh}{\operatorname{ch}}
\renewcommand{\sinh}{\operatorname{sh}}
\newcommand{\Avg}{\operatorname{Avg}}
\newcommand{\fM}{\mathfrak M}
\newcommand{\bC}{\mathbb C}
\newcommand{\bF}{\mathbf F}
\newcommand{\supp}{\operatorname{supp}}
\newcommand{\bI}{\mathbf I}
\newcommand{\bJ}{\mathbf J}

\newtheorem{theorem}{Theorem}
\newtheorem{lemma}[theorem]{Lemma}
\newtheorem{proposition}[theorem]{Proposition}

\numberwithin{equation}{section}
\numberwithin{theorem}{section}

\allowdisplaybreaks

\title[Endpoint estimates for hyperbolic spherical maximal operators]{Endpoint estimates for spherical maximal operators acting on radial functions in hyperbolic spaces}
\author{Yunxiang Wang}
\address{Yunxiang Wang, Department of Mathematics, Sun Yat-Sen University, Guangzhou, 510275, P.R. China}
\email{\href{mailto:wangyx677@mail.sysu.edu.cn}{wangyx677@mail.sysu.edu.cn}}

\begin{document}

    \begin{abstract}
        In this article we investigate spherical maximal operators $\mathfrak m^\alpha$ of real order $\alpha$ acting on radial functions in the hyperbolic space $\mathbb H^n$. We show that for every $n\geq 2$, $1<p\leq\infty$ and $\alpha>(1-n)/2$, $\mathfrak m^\alpha$ is bounded on $L^p_{\mathrm{rad}}(\mathbb H^n)$ if and only if
        \begin{align*}
            \begin{cases}
                \alpha>1-n+\frac np, & 1<p\leq 2,\\
                \alpha>\frac1p-\frac{n-1}2, & 2<p\leq\frac{2n}{n-1},\\
                \alpha\geq\frac{1-n}p, & p>\frac{2n}{n-1}.
            \end{cases}
        \end{align*}
        Moreover, we show that for $1<p\leq 2n/(n-1)$ and
        \begin{align*}
            \alpha=
            \begin{cases}
                1-n+\frac np, & 1<p\leq 2,\\
                \frac1p-\frac{n-1}2, & 2<p\leq\frac{2n}{n-1},
            \end{cases}
        \end{align*}
        $\mathfrak m^\alpha$ is bounded from $L^{p,1}_{\mathrm{rad}}(\mathbb H^n)$ to $L^{p,\infty}_{\mathrm{rad}}(\mathbb H^n)$.
    \end{abstract}

    \maketitle

    \section{Introduction}

    In the 1970s, Stein \cite{Ste76spherical} introduced the spherical maximal operators
    \begin{align*}
        \fM^\alpha(f)(x):=\sup_{t>0}\left|\frac1{\Gamma(\alpha)}\int_{\{y\in\bR^n:|y|<1\}}(1-|y|^2)^{\alpha-1}f(x-ty)\,\d y\right|.
    \end{align*}
    of complex order $\alpha$, with a priori for $\Re\alpha>0$, on $\bR^n$. This definition can be analytically continued to $\Re\alpha>(1-n)/2$. In particular, when $\alpha=0$, $\fM^\alpha$ recovers the standard spherical maximal operator $\fM$ defined by
    \begin{align*}
        \fM(f)(x):=\sup_{t>0}\left|\int_{\{\omega\in\bR^n:|\omega|=1\}}f(x-t\omega)\,\d\omega\right|,
    \end{align*}
    where $\d\omega$ is the canonical surface measure of the unit sphere in $\bR^n$. Stein (\cite[pp. 518--519]{Ste93book} \& \cite[Theorem 2]{Ste76spherical}) showed that for $n\geq 2$, $1<p<\infty$ and $\alpha\in\bR$ with $\alpha>(1-n)/2$, if
    \begin{align*}
        \alpha>
        \begin{cases}
            1-n+\frac np, & 1<p\leq 2,\\
            \frac{2-n}p, & p>2,
        \end{cases}
    \end{align*}
    then $\fM^\alpha$ is bounded on $L^p(\bR^n)$, and that if $\fM^\alpha$ is bounded on $L^p(\bR^n)$ for $1<p\leq 2$, then we must have
    \begin{align*}
        \alpha>1-n+\frac np.
    \end{align*}

    Stein's result with $\alpha=0$ indicates that for $n\geq 3$ and $p>n/(n-1)$, $\fM$ is bounded on $L^p(\bR^n)$, and that for $n\geq 2$, $\fM$ is bounded on $L^p(\bR^n)$ only if $p>n/(n-1)$. When $n=2$, the same sufficient condition $p>2=2/(2-1)$ for the $L^p$-boundedness of $\fM$ was obtained by Bourgain \cite{Bou86circular}. At the endpoint $p=n/(n-1)$, Bourgain \cite{Bou85spherical} showed that $\fM^\alpha$ is bounded from $L^{p,1}(\bR^n)$ to $L^{p,\infty}(\bR^n)$ when $n\geq 3$, and that the same restricted weak-type $(p,p)$ estimate holds for radial functions when $n=2$. However, this $2$-dimensional restricted weak-type estimate fails for general functions, as shown by Seeger, Tao and Wright \cite{STW03sphercial}. For further results concerning $L^p$-boundedness of the operators $\fM^\alpha$, one can refer to \cite{LSSY24spherical,MYZ17spherical,MSS92localsmoothing, Wol00localsmoothing} and the references therein.
    \smallskip

    In the last decade, Ciaurri et al. \cite{CNR17radial,CNR20radial,NRS23radial} investigated $\fM^\alpha$ acting on radial functions. They established the sharp range of $\alpha$ for $\fM^\alpha$ to be bounded on $L^p_{\rad}(\bR^n)$, the radial subspace of $L^p(\bR^n)$, as well as the corresponding endpoint result in terms of restricted weak-type estimates. More precisely, they showed that for $n\geq 2$, $1<p\leq \infty$ and $\alpha\in\bR$ with $\alpha>(1-n)/2$, $\fM^\alpha$ is bounded on $L^p_{\rad}(\bR^n)$ if and only if
    \begin{align*}
        \begin{cases}
            \alpha>1-n+\frac np, & 1<p\leq 2,\\
            \alpha>\frac1p-\frac{n-1}2, & 2<p\leq \frac{2n}{n-1},\\
            \alpha\geq \frac{1-n}p, & p>\frac{2n}{n-1},
        \end{cases}
    \end{align*}
    and that for $1<p\leq 2n/(n-1)$ and
    \begin{align*}
        \alpha=\begin{cases}
            1-n+\frac np, & 1<p\leq 2,\\
            \frac1p-\frac{n-1}2, & 2<p\leq\frac{2n}{n-1},
        \end{cases}
    \end{align*}
    $\fM^\alpha$ is bounded from $L^{p,1}_{\rad}(\bR^n)$ to $L^{p,\infty}_{\rad}(\bR^n)$. See \cite[Corollaries~1.5 and~1.6(b)]{NRS23radial}.
    \smallskip

    In the 1980s, El Kohen \cite{ElK80spherical} introduced on the hyperbolic space $\bH^n$ the spherical maximal operators
    \begin{align*}
        \fm^\alpha(f)(x):=\sup_{t>0}|M^\alpha_t(f)(x)|
    \end{align*}
    of complex order $\alpha$ with $\Re\alpha>0$, where
    \begin{align}\label{eq:S1sphavgdef}
        M^\alpha_t(f)(x):=\frac{(\sinh t)^{2-n}}{\Gamma(\alpha)\,(\cosh t-1)^\alpha}\int_{B(x,t)}\big(\cosh t-\cosh d(x,y)\big)^{\alpha-1}f(y)\,\d y
    \end{align}
    is the spherical averaging operators of order $\alpha$. Here $d(\cdot,\cdot)$ is the geodesic distance, $\d x$ is the Riemannian measure and $B(x,t)$ is the geodesic ball centered at $x$ with radius $t$ in $\bH^n$. The definition can be analytically continued to $\Re\alpha>(1-n)/2$. In particular, at order zero,
    \begin{align}\label{eq:S10order}
        M_t^0(f)(x)=|\bS^{n-1}|\Avg_{S(x,t)}(f),
    \end{align}
    where $S(x,t)$ is the geodesic sphere in $\bH^n$ and $\Avg_{S(x,t)}$ denotes its normalized surface average. El Kohen \cite[Theorem~3]{ElK80spherical} showed that for $n\geq 2$, $1<p<\infty$ and $\alpha\in\bR$ with $\alpha>(1-n)/2$, if
    \begin{align*}
        \alpha>
        \begin{cases}
            1-n+\frac np, & 1<p\leq 2,\\
            \frac{2-n}p, & p>2,
        \end{cases}
    \end{align*}
    then $\fm^\alpha$ is bounded on $L^p(\bH^n)$.

    This result with $\alpha=0$ yields that for $n\geq 3$ and $p>n/(n-1)$, $\fm^0$ is bounded on $L^p(\bH^n)$. The case of $n=2$ was later resolved by Ionescu \cite{Ion00FIO}. More recently, Chen, Shen, Yan and the named author \cite{Che+25spherical} gave the necessary condition for the $L^p$-boundedness of $\fm^\alpha$, and improved the range of $\alpha$ in \cite[Theorem~3]{ElK80spherical}  for $p>2$. They showed that for $n\geq 2$, $1<p<\infty$ and $\alpha\in\bR$ with $\alpha>(1-n)/2$, if $\fM^\alpha$ is bounded on $L^p(\bH^n)$, then we must have
    \begin{align*}
        \begin{cases}
            \alpha>1-n+\frac np, & 1<p\leq 2,\\
            \alpha\geq \frac1p-\frac{n-1}2, & 2<p\leq \frac{2n}{n-1},\\
            \alpha\geq \frac{1-n}p, & p>\frac{2n}{n-1},
        \end{cases}
    \end{align*}
    and that if $p>2$ and
    \begin{align*}
        \alpha>
        \begin{cases}
            \frac{2-n}p-\frac{p-2}{pp_n(p_n-2)}, & 2<p\leq p_n,\\
            \frac{2-n}p-\frac1{pp_n}, & p>p_n
        \end{cases}
    \end{align*}
    with
    \begin{align*}
        p_n=
        \begin{cases}
            4, & n=2,\\
            \frac{2(n+1)}{n-1}, & n\geq 3,
        \end{cases}
    \end{align*}
    then $\fM^\alpha$ is bounded on $L^p(\bH^n)$. See \cite[Theorems~1.1 and~1.2]{Che+25spherical}. In addition, Zhang and the named author \cite{WZ26lacunary} investigated the related problems of maximal operators with the supremum restricted to lacunary radii.
    \smallskip

    In this article we focus on $\fm^\alpha$ acting on \emph{radial} functions, which depend only on $d(x,\0)$ for $\0$ being the origin in $\bH^n$. Our main result is the following.
    \begin{theorem}\label{thm:main}
        Suppose $n\geq 2$, $1<p\leq\infty$ and $\alpha\in\bR$ with $\alpha>(1-n)/2$.
        \begin{asparaenum}[\rm (i)]
            \item\label{thm:main(i)} $\fm^\alpha$ is bounded on $L^p_{\rad}(\bH^n)$ if and only if
            \begin{align*}
                \begin{cases}
                    \alpha>1-n+\frac np, & 1<p\leq 2,\\
                    \alpha>\frac1p-\frac{n-1}2, & 2<p\leq \frac{2n}{n-1},\\
                    \alpha\geq \frac{1-n}p, & p>\frac{2n}{n-1}.
                \end{cases}
            \end{align*}
            \item\label{thm:main(ii)} If
            \begin{align*}
                \alpha=\begin{cases}
                    1-n+\frac np, & 1<p\leq 2,\\
                    \frac1p-\frac{n-1}2, & 2<p\leq\frac{2n}{n-1},
                \end{cases}
            \end{align*}
            then $\fm^\alpha$ is bounded from $L^{p,1}_{\rad}(\bH^n)$ to $L^{p,\infty}_{\rad}(\bH^n)$.
        \end{asparaenum}
    \end{theorem}
    The relations between $\alpha$ and $p$ in Theorem~\ref{thm:main} are indicated in Figure~\ref{fig:1} below.
    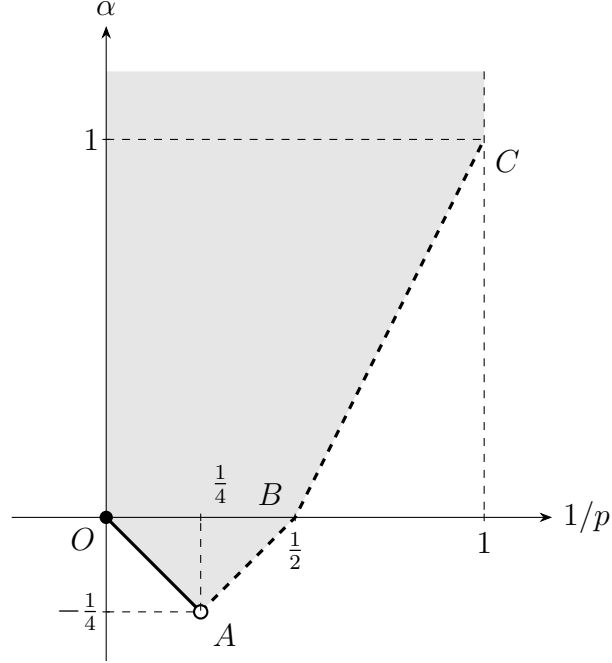
\begin{figure}[H]
        \centering
        \begin{tikzpicture}[scale=5,>=Stealth]
            \fill[black!10] (0,0)--(.25,-.25)--(.5,0)--(1,1)--(1,1.18)--(0,1.18)--cycle;
            \draw[->] (-.25,0)--(1.18,0) node[right] {$1/p$};
            \draw[->] (0,-.38)--(0,1.3) node[above] {$\alpha$};
            \draw[black,very thick] (0,0)--(.25,-.25) node[below right] {$A$};
            \draw[black,very thick,dashed] (.25,-.25)--(.5,0) node[above left] {$B$}--(1,1);
            \filldraw[fill=white,draw=black,thick] (.25,-.25) circle (0.5pt);
            \fill[black] (0,0) circle (0.5pt);
            \draw (.25,-.01)--(.25,0.01) node[above right] {$\frac14$};
            \draw (.5,.01)--(.5,-0.01) node[below] {$\frac12$};
            \draw (1,.01)--(1,-0.01) node[below] {$1$};
            \draw (-0.01,-.25)--(.01,-.25) node[left] {$-\frac14$};
            \draw (-0.01,1)--(.01,1) node[left] {$1$};
            \node[below left] at (0,0) {$O$};
            \draw[dashed] (0,-.25)--(.25,-.25)--(.25,0);
            \draw[dashed] (0,1)--(1,1) node[below right] {$C$};
            \draw[dashed] (1,1.18)--(1,0);
        \end{tikzpicture}
        \caption{The $(1/p,\alpha)$-plot for $n=2$. $\fm^\alpha$ is bounded on $L^p_{\rad}(\bH^n)$ if and only if $(1/p,\alpha)$ is on or above the solid segment $OA$ (not containing $A$), or strictly above the dashed folded segments $ABC$. In addition, if $(1/p,\alpha)$ is on the dashed folded segments $ABC$ (containing $A$), $\fm^\alpha$ is bounded from $L^{p,1}_{\rad}(\bH^n)$ to $L^{p,\infty}_{\rad}(\bH^n)$.}
        \label{fig:1}
    \end{figure}
    \smallskip

    We would like to mention that according to the arguments in \cite{Che+25spherical}, the threshold of the range of $\alpha$ in Theorem~\ref{thm:main}\eqref{thm:main(i)} corresponds to both the threshold of the Lebesgue index and the growth in time in the local smoothing estimates for wave equation with radial initial datum, obtained by Chen, Duong, Yan and the named author \cite{CDWY26radial}. In addition, both the strong-type and the restricted weak-type boundedness coincide with its Euclidean counterpart \cite[Corollaries~1.5 and~1.6(b)]{NRS23radial}, indicating that the geometry at infinity induced by the negative curvature of $\bH^n$ imposes no additional restriction on the $L^p_{\rad}$-boundedness. Based on this observation, we conjecture that the sufficient condition for the general $L^p$-boundedness of $\fm^\alpha$ in \cite[Theorems~1.1 and~1.2]{Che+25spherical} can be sharpened so that there is no gap between them.

    To show Theorem~\ref{thm:main}, we first derive the explicit radial kernel expressions of the operator $M^\alpha_t$ in terms of the Gauss hypergeometric function. We then decompose the operator $\fm^\alpha$ into local and global parts, with supremums taken over all small and large radii, respectively. The local part can be essentially dominated by the maximal operators with positive Euclidean radial kernels studied in \cite{CNR17radial,CNR20radial,NRS23radial}. For the global part, we continue splitting the supremum into unit intervals and use the exponential decay of the kernel to obtain summable estimates.

    The rest of this article is organized as follows. Section~\ref{sec:prelim} recalls basic facts about hyperbolic spaces and establishes the explicit radial kernel expressions and estimates. With these preparations, in Sections~\ref{sec:proof} we prove Theorem~\ref{thm:main}.

    \section{Preliminaries}\label{sec:prelim}

    In this section, we recall basic facts about the hyperbolic space $\bH^n$ and obtain the explicit radial kernel expressions of the spherical maximal operators.

    \subsection{\texorpdfstring{$\bH^n$}{Hn} as the hyperboloid}

    On $\bR^{1+n}$ we define the Minkowski metric
    \begin{align*}
        [x,y]=x_0y_0-x'\cdot y',\quad(x_0,x'),\,(y_0,y')\in\bR^{1+n}.
    \end{align*}
    The hyperboloid model is
    \begin{align*}
        \bH^n:=\{x=(x_0,x')\in \bR^{1+n}:\ [x,x]=1,\, x_0>0\}.
    \end{align*}
    For $x\in\bH^n$, its tangent space is
    \begin{align*}
        T_x\bH^n:=\{v\in\bR^{1+n}:\ [v,x]=0\}.
    \end{align*}
    The Riemannian metric $\bg$ is given by the bilinear form
    \begin{align*}
        \langle v,w\rangle_x:=-[v,w],\quad v,w\in T_x\bH^n.
    \end{align*}
    This metric gives the complete simply connected space of constant sectional curvature $-1$. Its geodesic distance is given by
    \begin{align*}
        d(x,y):=\arcosh([x,y]).
    \end{align*}

    Take $\0:=(1,0,\dots,0)\in \bH^n$. The geodesic polar coordinates at $\0$ are
    \begin{align*}
        x=(\cosh R,\sinh R\, \theta),\quad R>0,\ \theta\in\bS^{n-1}.
    \end{align*}
    In these coordinates the Riemannian metric reads
    \begin{align*}
        \bg=\d R^2+(\sinh R)^2\bg_{\bS^{n-1}}.
    \end{align*}
    The induced Riemannian measure is therefore given by
    \begin{align}\label{eq:S2Riemmeas}
        \d x=(\sinh R)^{n-1}\d R\d \theta.
    \end{align}

    If $x=(\cosh R,\sinh R\, \theta)$ and $y=(\cosh r,\sinh r\, \omega)$,
    then we have the hyperbolic cosine law:
    \begin{align}\label{eq:S2coslaw}
        \cosh d(x,y)=\cosh R\cosh r-\sinh R\sinh r\,\theta\cdot\omega.
    \end{align}

    Radial functions on $\bH^n$ depend on $R(x):=d(x,\0)$ only. In other words, for such a function $f$, there exists another function $h$, namely the \emph{radial profile} of $f$, defined on $(0,\infty)$, such that
    \begin{align}\label{eq:S2radfunct}
        f(x)=h(R(x))\quad\text{for }x\ne \0.
    \end{align}
    In the sequel, we always identify $f$ with its radial profile. Note that in view of \eqref{eq:S2Riemmeas},
    \begin{align}\label{eq:S2radialintegral}
        \int_{\bH^n}f(x)\,\d x=\nu_{n-1}\int_0^\infty f(r)\,(\sinh r)^{n-1}\d r,
    \end{align}
    where $\nu_{n-1}=|\bS^{n-1}|$.

    \subsection{Explicit kernel expressions for \texorpdfstring{$M^\alpha_t$}{Mαt}}

    For $\Re c>\Re b>0$ and $|z|<1$, let
    \begin{align}\label{eq:S2hypergeom}
        F(a,b;c;z):=\frac{\Gamma(c)}{\Gamma(b)\,\Gamma(c-b)}\int_0^1 (1-\tau)^{c-b-1}\tau^{b-1}(1-z\tau)^{-a}\d \tau
    \end{align}
    be the Gauss hypergeometric function. We further define
    \begin{align}\label{eq:S2bF}
        \bF(a,b;c;z):=\frac{F(a,b;c;z)}{\Gamma(c)}.
    \end{align}
    We need the following analytic facts about the functions $F$ and $\bF$.
    \begin{lemma}\label{lem:S2hypfunfacts}
        Suppose $z\in\bC$ with $|z|<1$.
        \begin{asparaenum}[\rm (1)]
            \item\label{lem:S2hypfunfacts(1)} For $\Re c>\Re b>0$ the function $F(a,b;c;z)$ is an entire function in $a$.
            \item\label{lem:S2hypfunfacts(2)} The function $\bF(a,b;c;z)$ is an entire function in $c$.
            \item\label{lem:S2hypfunfacts(3)} The functions $F$ and $\bF$ satisfy
            \begin{align*}
                |F(a,b;c;z)|+|\bF(a,b;c;z)|\lesssim
                \begin{cases}
                    (1-|z|)^{\Re(c-a-b)}, & a+b>c,\\
                    1+\log\frac1{1-|z|}, & a+b=c,\\
                    1, & a+b<c.
                \end{cases}
            \end{align*}
            In particular, for $\Re(c-a-b)>-1$, $F$ and $\bF$ are both locally integrable in $z$.
        \end{asparaenum}
    \end{lemma}
    \begin{proof}
        \eqref{lem:S2hypfunfacts(1)} follows readily from \eqref{eq:S2hypergeom} and Morera's Theorem.

        To show \eqref{lem:S2hypfunfacts(2)}, it follows from \cite[(2) and (5) on pp. 56--57]{Erd+53transcendentalI} that
        \begin{align*}
            \bF(a,b;c;z)=\sum_{j=0}^\infty\frac{\Gamma(a+j)}{\Gamma(a)}\frac{\Gamma(b+j)}{\Gamma(b)}\frac{z^j}{j!\,\Gamma(c+j)}.
        \end{align*}
        Since for each $j=0,1\dots$, $1/\Gamma(\cdot+j)$ is entire, for $|z|<1$ the above function is entire in $c$.

        To show \eqref{lem:S2hypfunfacts(3)}, we apply \cite[(2) and (5) on pp. 56--57]{Erd+53transcendentalI} to derive
        \begin{align*}
            |F(a,b;c;z)|+|\bF(a,b;c;z)|\lesssim \sum_{j=0}^\infty (1+j)^{\Re(a+b-c)-1}|z|^j.
        \end{align*}
        This series is dominated by the desired estimate, so are $F(a,b;c;z)$ and $\bF(a,b;c;z)$.
    \end{proof}

    Recall that for $\Re\alpha>0$, the spherical means $M^\alpha_t$ on $\bH^n$ are defined by \eqref{eq:S1sphavgdef}. We have the following explicit expression.
    \begin{lemma}\label{lem:S2kerexp}
        Suppose $n\geq 2$ and $\alpha\in\bC$. For a smooth compactly supported radial function $f$, $M^\alpha_t(f)$ can be analytically continued to $\Re\alpha>(1-n)/2$ and is also radial. We can write
        \begin{align*}
            M^\alpha_t(f)(R)=\int_0^\infty K^\alpha_t(R,r)\, f(r)\, (\sinh r)^{n-1}\d r,
        \end{align*}
        where for $t,R,r>0$ away from $t=|R-r|$ and $t=R+r$, $K^\alpha_t$ is given as follows:
        \begin{asparaenum}[\rm (i)]
            \item\label{lem:S2kerexp(i)} If $t<|R-r|$, then $K^\alpha_t(R,r)=0$.
            \item\label{lem:S2kerexp(ii)} If $|R-r|<t<R+r$, then
            \begin{align*}
                K^\alpha_t(R,r)={}&2^{\frac{n-3}2}\nu_{n-2}\Gamma\!\left(\frac{n-1}2\right)\frac{(\sinh t)^{2-n}}{(\cosh t-1)^\alpha}(\sinh R\, \sinh r)^{\frac{1-n}2}(\cosh t-\cosh(R-r))^{\alpha+\frac{n-3}2}\\
                {}&\times \bF\!\left(\frac{3-n}2,\frac{n-1}2;\alpha+\frac{n-1}2;\frac{\cosh t-\cosh(R-r)}{2\sinh R\,\sinh r}\right).
            \end{align*}
            \item\label{lem:S2kerexp(iii)} If $t>R+r$, then
            \begin{align*}
                K^\alpha_t(R,r)={}&\frac{2^{n-2}\nu_{n-2}\Gamma(\frac{n-1}2)^2}{\Gamma(\alpha)\,\Gamma(n-1)}\frac{(\sinh t)^{2-n}}{(\cosh t-1)^\alpha}(\cosh t-\cosh (R-r))^{\alpha-1}\\
                {}&\times F\!\left(1-\alpha,\frac{n-1}2;n-1;\frac{2\sinh R\,\sinh r}{\cosh t-\cosh(R-r)}\right).
            \end{align*}
        \end{asparaenum}
        Moreover, for $\Re\alpha>(1-n)/2$, the values on $t=|R-r|$ and $t=R+r$ may be assigned arbitrarily.
    \end{lemma}
    \begin{proof}
        Assume first $\Re\alpha>0$. In view of \eqref{eq:S1sphavgdef}, \eqref{eq:S2radialintegral} and \eqref{eq:S2coslaw}, for $x=(\cosh R,\sinh R\,\theta)$,
        \begin{align*}
            {}&M^\alpha_t(f)(x)\\
            ={}&\frac{(\sinh t)^{2-n}}{\Gamma(\alpha)\,(\cosh t-1)^\alpha}\int_0^\infty f(r)\, (\sinh r)^{n-1}\\
            {}&\times\left(\int_{\bS^{n-1}}(\cosh t-\cosh R\,\cosh r+\sinh R\,\sinh r\,\theta\cdot\omega)_+^{\alpha-1}\d\omega\right)\d r\\
            ={}&\frac{\nu_{n-2}(\sinh t)^{2-n}}{\Gamma(\alpha)\,(\cosh t-1)^\alpha}\int_0^\infty f(r)\, (\sinh r)^{n-1}\\
            {}&\times\left(\int_{-1}^1(\cosh t-\cosh R\,\cosh r+\sinh R\,\sinh r\,w)_+^{\alpha-1}(1-w^2)^{\frac{n-3}2}\d w\right)\d r\\
            ={}&\frac{\nu_{n-2}(\sinh t)^{2-n}}{\Gamma(\alpha)\,(\cosh t-1)^\alpha}\int_0^\infty f(r)\, (\sinh r)^{n-1}(\sinh R\,\sinh r)^{2-n}\\
            {}&\times\left(\int_{\cosh(R-r)}^{\cosh(R+r)}(\cosh t-v)_+^{\alpha-1}\big(v-\cosh(R-r)\big)^{\frac{n-3}2}\big(\cosh(R+r)-v\big)^{\frac{n-3}2}\d v\right)\d r,
        \end{align*}
        where $(A)_+^{\alpha-1}$ denotes $A^{\alpha-1}\1_{A>0}$. Hence
        \begin{align}\label{eq:S2kerintexp}
            K^\alpha_t(R,r)={}&\frac{\nu_{n-2}(\sinh t)^{2-n}}{\Gamma(\alpha)\,(\cosh t-1)^\alpha}(\sinh R\,\sinh r)^{2-n}\\
            {}&\times\int_{\cosh(R-r)}^{\cosh(R+r)}(\cosh t-v)_+^{\alpha-1}\big(v-\cosh(R-r)\big)^{\frac{n-3}2}\big(\cosh(R+r)-v\big)^{\frac{n-3}2}\d v.\notag
        \end{align}

        Now suppose $\Re\alpha>(1-n)/2$. \eqref{lem:S2kerexp(i)} follows from analytic continuation and the fact that if $\cosh t<\cosh(R-r)$ then the integrand in \eqref{eq:S2kerintexp} is zero.

        To see \eqref{lem:S2kerexp(ii)}, note that if $\cosh(R-r)<\cosh t<\cosh (R+r)$, letting
        \begin{align*}
            s:=\frac{v-\cosh(R-r)}{\cosh t-\cosh(R-r)}
        \end{align*}
        in the integral in \eqref{eq:S2kerintexp}, we have
        \begin{align*}
            K^\alpha_t(R,r)={}&\frac{2^{\frac{n-3}2}\nu_{n-2}(\sinh t)^{2-n}}{\Gamma(\alpha)\,(\cosh t-1)^\alpha}(\sinh R\,\sinh r)^{\frac{1-n}2}\big(\cosh t-\cosh(R-r)\big)^{\alpha+\frac{n-3}2}\\
            {}&\times \int_0^1(1-s)^{\alpha-1}s^{\frac{n-3}2}\left(1-\frac{\cosh t-\cosh(R-r)}{2\sinh R\,\sinh r}s\right)^{\frac{n-3}2}\d s.
        \end{align*}
        In view of \eqref{eq:S2hypergeom},
        \begin{align*}
            K^\alpha_t(R,r)={}&2^{\frac{n-3}2}\nu_{n-2}\Gamma\!\left(\frac{n-1}2\right)\frac{(\sinh t)^{2-n}}{(\cosh t-1)^\alpha}(\sinh R\, \sinh r)^{\frac{1-n}2}(\cosh t-\cosh(R-r))^{\alpha+\frac{n-3}2}\\
            {}&\times \frac1{\Gamma(\alpha+\frac{n-1}2)}F\!\left(\frac{3-n}2,\frac{n-1}2;\alpha+\frac{n-1}2;\frac{\cosh t-\cosh(R-r)}{2\sinh R\,\sinh r}\right).
        \end{align*}
        It remains to invoke \eqref{eq:S2bF} to obtain the desired expression. Then one uses Lemma~\ref{lem:S2hypfunfacts}\eqref{lem:S2hypfunfacts(2)} to analytically continue $K^\alpha_t$ to $\Re\alpha>(1-n)/2$.

        To see \eqref{lem:S2kerexp(iii)}, note that if $\cosh t>\cosh(R+r)$, letting
        \begin{align*}
            s=\frac{v-\cosh(R-r)}{2\sinh R\,\sinh r}
        \end{align*}
        in the integral in \eqref{eq:S2kerintexp}, we have
        \begin{align*}
            K^\alpha_t(R,r)={}&\frac{2^{n-2}\nu_{n-2}(\sinh t)^{2-n}}{\Gamma(\alpha)\,(\cosh t-1)^\alpha}\big(\cosh t-\cosh(R-r)\big)^{\alpha-1}\\
            {}&\times \int_0^1(1-s)^{\frac{n-3}2}s^{\frac{n-3}2}\left(1-\frac{2\sinh R\,\sinh r}{\cosh t-\cosh(R-r)}s\right)^{\alpha-1}\d s.
        \end{align*}
        In view of \eqref{eq:S2hypergeom}, we obtain the desired expression. Then one uses Lemma~\ref{lem:S2hypfunfacts}\eqref{lem:S2hypfunfacts(1)} to analytically continue $K^\alpha_t$ to $\Re\alpha>(1-n)/2$.

        Finally, for $\Re\alpha>(1-n)/2$, the expression in \eqref{lem:S2kerexp(ii)} as well as Lemma~\ref{lem:S2hypfunfacts}\eqref{lem:S2hypfunfacts(3)} indicates that $K^\alpha_t$ is a locally integrable function on $(0,\infty)\times(0,\infty)$. Thus the values on these surfaces can be assigned arbitrarily.
    \end{proof}

    Lemmas~\ref{lem:S2hypfunfacts}\eqref{lem:S2hypfunfacts(3)} and~\ref{lem:S2kerexp}\eqref{lem:S2kerexp(ii)}\eqref{lem:S2kerexp(iii)} immediately yield the following crucial estimates.
    \begin{lemma}\label{lem:S2kerest}
        Suppose $\alpha\in\bR$ with $\alpha>(1-n)/2$. If $|R-r|<t<R+r$,
        \begin{align}
            |K_t^\alpha(R,r)|\lesssim{}&\frac{(\sinh t)^{2-n}}{(\cosh t-1)^\alpha}(\sinh R\,\sinh r)^{\frac{1-n}2}(\cosh t-\cosh(R-r))^{\alpha+\frac{n-3}2}\label{eq:kerest(i)}\\
            {}&\times\begin{cases}
                \big(\frac{\cosh(R+r)-\cosh t}{\sinh R\,\sinh r}\big)^{\alpha+\frac{n-3}2},&\frac{1-n}2<\alpha<\frac{3-n}2,\\
                1+\log\frac{2\sinh R\,\sinh r}{\cosh(R+r)-\cosh t},&\alpha=\frac{3-n}2,\\
                1,&\alpha>\frac{3-n}2.\notag
            \end{cases}
        \end{align}
        If $t>R+r$,
        \begin{align}
            |K_t^\alpha(R,r)|\lesssim{}&\frac{(\sinh t)^{2-n}}{(\cosh t-1)^\alpha}(\cosh t-\cosh(R-r))^{\alpha-1}\label{eq:kerest(ii)}\\
            {}&\times\begin{cases}
                \big(\frac{\cosh t-\cosh(R+r)}{\cosh t-\cosh(R-r)}\big)^{\alpha+\frac{n-3}2}, & \frac{1-n}2<\alpha<\frac{3-n}2,\\
                1+\log\frac{\cosh t-\cosh(R-r)}{\cosh t-\cosh(R+r)}, & \alpha=\frac{3-n}2,\\
                1, & \alpha>\frac{3-n}2.\notag
            \end{cases}
        \end{align}
    \end{lemma}

    We conclude this section by showing the following result, which holds for general $C^\infty_c(\bH^n)$ functions.
    \begin{lemma}\label{lem:S2raising}
        For $\alpha>(1-n)/2$, $\delta>0$ and $f\in C^\infty_c(\bH^n)$,
        \begin{align*}
            \fm^{\alpha+\delta}(f)(x)\lesssim\fm^\alpha(f)(x).
        \end{align*}
    \end{lemma}
    \begin{proof}
        In view of the proof of \cite[Lemma 1]{ElK80spherical},
        \begin{align*}
            M^{\alpha+\delta}_t(f)(x)=\frac1{\Gamma(\delta)}\int_0^1 M^\alpha_{\tau(t,u)}(f)(x)\left(\frac{\sinh\tau(t,u)}{\sinh t}\right)^{n-2}u^\alpha(1-u)^{\delta-1}\d u,
        \end{align*}
        where the function $\tau(t,u)\geq0$ is determined by
        \begin{align*}
            \left(\sinh\frac{\tau(t,u)}2\right)^2=u\left(\sinh\frac t2\right)^2.
        \end{align*}
        Note that $\sinh \tau(t,u)\lesssim u^{1/2}\sinh t$, whence
        \begin{align*}
            \fm^{\alpha+\delta}(f)(x)\lesssim \fm^\alpha(f)(x)\int_0^1u^{\alpha+\frac{n-2}2}(1-u)^{\delta-1}\d u\lesssim \fm^\alpha(f)(x),
        \end{align*}
        as desired.
    \end{proof}

    \section{Proof of Theorem~\ref{thm:main}}\label{sec:proof}

    In this section, we decompose the maximal operator $\fm^\alpha$ into the local part
    \begin{align*}
        \fm_0^\alpha (f)(R)=\sup_{0<t\leq 1}|M_t^\alpha (f)(R)|
    \end{align*}
    as well as the global parts
    \begin{align*}
        \fm_T^\alpha (f)(R)=\sup_{T\leq t\leq T+1}|M_t^\alpha (f)(R)|,
        \qquad T=1,2,\ldots.
    \end{align*}
    We then separately consider their $L^p_{\rad}$-boundedness.

    \subsection{The local part \texorpdfstring{$\fm^\alpha_0$}{mα0}}

    The aim of this subsection is to show the following result.
    \begin{proposition}\label{prop:SS3.1local}
        Suppose $1<p<\infty$, $\alpha<1/p$ and
        \begin{equation}\label{eq:SS3.1threeweak}
            \alpha\geq
            \begin{cases}
                1-n+\frac np, & 1<p\leq 2,\\
                \frac1p-\frac{n-1}2, & 2<p\leq \frac{2n}{n-1},\\
                \frac{1-n}p, & p>\frac{2n}{n-1}.
            \end{cases}
        \end{equation}
        Then $\fm^\alpha_0$ is bounded from $L^{p,1}_{\rad}(\bH^n)$ to $L^{p,\infty}_{\rad}(\bH^n)$.

        Furthermore, if $p>2n/(n-1)$ and $\alpha=(1-n)/p$, then $\fm^\alpha_0$ is bounded on $L^p_{\rad}(\bH^n)$.
    \end{proposition}

    To prove Proposition~\ref{prop:SS3.1local}, we define the Euclidean radial kernel $E^\alpha_t(R,r)$ as follows:
    \begin{asparaenum}[(i)]
        \item If $t<|R-r|$, then $E^\alpha_t(R,r)=0$.
        \item If $|R-r|<t<R+r$, then
        \begin{align*}
            E^\alpha_t(R,r)={}&\frac{(Rr)^{\frac{1-n}2}}{t^{n-2+2\alpha}}(t^2-(R-r)^2)^{\alpha+\frac{n-3}2}\times
            \begin{cases}
                \big(\frac{(R+r)^2-t^2)}{Rr}\big)^{\alpha+\frac{n-3}2}, & \frac{1-n}2<\alpha<\frac{3-n}2,\\
                1+\log\!\big(\frac{4Rr}{(R+r)^2-t^2}\big), & \alpha=\frac{3-n}2,\\
                1, & \alpha>\frac{3-n}2.\\
            \end{cases}
        \end{align*}
        \item If $t>R+r$, then
        \begin{align*}
            E^\alpha_t(R,r)=\frac1{t^{n-2+2\alpha}}(t^2-(R-r)^2)^{\alpha-1}\times
            \begin{cases}
                \big(\frac{t^2-(R+r)^2}{t^2-(R-r)^2}\big)^{\alpha+\frac{n-3}2}, & \frac{1-n}2<\alpha<\frac{3-n}2,\\
                1+\log\!\big(\frac{t^2-(R-r)^2}{t^2-(R+r)^2}\big), & \alpha=\frac{3-n}2,\\
                1, & \alpha>\frac{3-n}2.
            \end{cases}
        \end{align*}
    \end{asparaenum}
    We further define the associated maximal function
    \begin{align*}
        \fe^\alpha(f)(R):=\sup_{t>0}\int_0^\infty E^\alpha_t(R,r)|f(r)|\,r^{n-1}\d r.
    \end{align*}
    The following result concerning the $L^p_{\rad}$-boundedness of $\fe^\alpha$ serves as a key ingredient in the proof of Proposition~\ref{prop:SS3.1local}.
    \begin{lemma}\label{lem:SS3.1Eumaxest}
        Let $1<p<\infty$, $(1-n)/2<\alpha<1/p$. If \eqref{eq:SS3.1threeweak} holds, then $\fe^\alpha$ is bounded from $L^{p,1}(\bR_+,r^{n-1}\d r)$ to $L^{p,\infty}(\bR_+,r^{n-1}\d r)$. Furthermore, if $p>2n/(n-1)$ and $\alpha=(1-n)/p$, then
        $\fe^\alpha$ is bounded on $L^p(\bR_+,r^{n-1}\d r)$.
    \end{lemma}
    \begin{proof}
        This follows from the proofs of sufficiency part of \cite[Theorems~A and~B(b)]{NRS23radial}, see \cite[Sections 3--4]{NRS23radial}.
    \end{proof}

    Now we present the
    \begin{proof}[Proof of Proposition~\ref{prop:SS3.1local}]
        Fix a radial function $f$. We write
        \begin{align}
            \fm^\alpha_0(f)={}&\1_{(0,2]}\fm^\alpha_0(f)+\sum_{j=2}^\infty\1_{(j,j+1]}\fm^\alpha_0(f)\label{eq:SS3.1decom}\\
            ={}&\underbrace{\1_{(0,2]}\fm^\alpha_0(f\1_{(0,3]})}_{\bI^\alpha_1(f)}+\sum_{j=2}^\infty\underbrace{\1_{(j,j+1]}\fm^\alpha_0(f\1_{(j-1,j+2]})}_{\bI^\alpha_j(f)},\notag
        \end{align}
        where the second equality follows from Lemma~\ref{lem:S2kerexp}\eqref{lem:S2kerexp(i)}.
        \smallskip

        We begin by estimating $\bI^\alpha_1(f)$. Since
        \begin{align*}
            \cosh a-\cosh b=\frac{\e^a}2(1-\e^{-a+b})(1-\e^{-a-b})
        \end{align*}
        Lemma~\ref{lem:S2kerest} yields that for $0<t\leq 1$, $R\leq 2$ and $r\leq 3$, we have
        \begin{equation*}
            |K_t^\alpha(R,r)|\lesssim E^\alpha_t(R,r).
        \end{equation*}
        Therefore, in view of Lemma~\ref{lem:SS3.1Eumaxest}, $\bI^\alpha_1$ satisfies the desired boundedness properties.
        \smallskip

        Fix $j=2,3,\dots$. It remains to estimate $\bI^\alpha_j(f)$. For $0<t\leq 1$, $R\in (j,j+1]$ and $r\in (j-1,j+2]$, we have
        \begin{align*}
            \cosh t-\cosh(R-r)\sim t^2-(R-r)^2,\qquad \cosh(R+r)-\cosh t\sim\sinh R\,\sinh r.
        \end{align*}
        Thus Lemma~\ref{lem:S2kerest} yields
        \begin{align*}
            |K_t^\alpha(R,r)|(\sinh r)^{n-1}\lesssim t^{2-n-2\alpha}\big(t^2-(R-r)^2\big)^{\alpha+\frac{n-3}2},
        \end{align*}
        whence
        \begin{align*}
            \fm^\alpha_0(f)(R)&\lesssim \sup_{0<t\leq 1}t^{2-n-2\alpha}\int_{-t}^t
            (t^2-r^2)^{\alpha+\frac{n-3}2}|f(R-r)|\,\d r\\
            &=\sup_{0<t\leq 1}\int_{-1}^1(1-r^2)^{\alpha+\frac{n-3}2}|f(R-tr)|\,\d r.
        \end{align*}
        It follows from \cite[Theorems~A and~B(b)]{NRS23radial} that $\bI^\alpha_j$ is bounded from $L^p((j-1,j+2],\d r)$ to $L^p((j,j+1],\d R)$ when $\alpha>1/p-(n-1)/2$, and from $L^{p,1}((j-1,j+2],\d r)$ to $L^{p,\infty}((j,j+1],\d R)$. These estimates in turn imply
        \begin{align*}
            \|\bI^\alpha_j(f)\|_{L^p((j,j+1],(\sinh R)^{n-1}\d R)}\lesssim \|f\|_{L^p((j-1,j+2],(\sinh R)^{n-1}\d r)}\qquad\text{for }\alpha>\frac1p-\frac{n-1}2,
        \end{align*}
        and
        \begin{align*}
            \|\bI^\alpha_j(f)\|_{L^{p,\infty}((j,j+1],(\sinh R)^{n-1}\d R)}\lesssim \|f\|_{L^{p,1}((j-1,j+2],(\sinh r)^{n-1}\d r)}\qquad\text{for }\alpha\geq \frac1p-\frac{n-1}2.
        \end{align*}
        By the disjointness of $(j,j+1]$ and finite overlapping of $(j-1,j+2]$, we sum over $j=2,3,\dots$ to obtain
        \begin{align}
            \left\|\sum_{j=2}^\infty\bI^\alpha_j(f)\right\|_{L^p_{\rad}(\bH^n)}^p={}&\sum_{j=2}^\infty\|\bI^\alpha_j(f)\|_{L^p((j,j+1],(\sinh R)^{n-1}\d R)}^p\label{eq:SS3.1strong}\\
            \lesssim{}&\sum_{j=2}^\infty\|f\|_{L^p((j-1,j+2],(\sinh r)^{n-1}\d r)}^p\lesssim \|f\|_{L^p_{\rad}(\bH^n)}^p\qquad\text{for }\alpha>\frac1p-\frac{n-1}2,\notag
        \end{align}
        while in view of \cite[Proposition~1.4.9]{Gra14GTM249}
        \begin{align}
            {}&\left\|\sum_{j=2}^\infty\bI^\alpha_j(f)\right\|_{L^{p,\infty}_{\rad}(\bH^n)}^p\label{eq:SS3.1resweak}\\
            \leq{}&\sum_{j=2}^\infty\left\|\bI^\alpha_j(f)\right\|_{L^{p,\infty}((j-1,j],(\sinh R)^{n-1}\d R)}^p\notag\\
            \lesssim{}&\sum_{j=2}^\infty\|f\|_{L^{p,1}((j-1,j+2],(\sinh r)^{n-1}\d r)}^p\notag\\
            \sim{}&\sum_{j=2}^\infty\left(\int_0^\infty\left(\int_{j-1}^{j+2}\1_{\{|f(r)|>s\}}(\sinh r)^{n-1}\d r\right)^{\frac1p}\d s\right)^p\notag\\
            \leq{}&\left(\int_0^\infty\left(\sum_{j=2}^\infty\int_{j-1}^{j+2}\1_{\{|f(r)|>s\}}(\sinh r)^{n-1}\d r\right)^{\frac1p}\d s\right)^p\notag\\
            \lesssim{}&\left(\int_0^\infty\left(\int_0^\infty\1_{\{|f(r)|>s\}}(\sinh r)^{n-1}\d r\right)^{\frac1p}\d s\right)^p\sim \|f\|_{L^{p,1}_{\rad}(\bH^n)}^p\qquad\text{for }\alpha\geq\frac1p-\frac{n-1}2.\notag
        \end{align}
        \smallskip

        According to \eqref{eq:SS3.1decom}, these estimates yield the desired boundedness of $\fm^\alpha_0$.
    \end{proof}

    \subsection{The global parts \texorpdfstring{$\fm^\alpha_T$}{mαT}
        (\texorpdfstring{$T=1,2,\dots$}{T=1,2,...})}

    The aim of this subsection is to show the following result.
    \begin{proposition}\label{prop:SS3.2global}
        Suppose $1<p<\infty$, $T=1,2,\ldots$, $\alpha<1/p$ and \eqref{eq:SS3.1threeweak} holds. Set
        \begin{align*}
            c_p=(n-1)\min\!\left\{\frac1p,\frac1{p'}\right\}>0.
        \end{align*}
        Then
        \begin{align*}
            \|\fm^\alpha_T(f)\|_{L^{p,\infty}_{\rad}(\bH^n)}\lesssim T\,\e^{-c_pT}\|f\|_{L^{p,1}_{\rad}(\bH^n)}.
        \end{align*}

        Furthermore, if $p>2n/(n-1)$ and $\alpha=(1-n)/p$, then the above estimate also holds with $L^p_{\rad}$ in place of $L^{p,\infty}_{\rad}$ and $L^{p,1}_{\rad}$.
    \end{proposition}
    \begin{proof}
        We first write
        \begin{align}\label{eq:SS3.2decom}
            \fm^\alpha_T(f)\leq \underbrace{\1_{(\frac1{10},\infty)}\fm^\alpha_T\big(f\1_{(\frac1{10},\infty)}\big)}_{\bJ^\alpha_1(f)}+\underbrace{\1_{(\frac1{10},\infty)}\fm^\alpha_T\big(f\1_{(0,\frac1{10}]}\big)}_{\bJ^\alpha_2(f)}+\underbrace{\1_{(0,\frac1{10}]}\fm^\alpha_T(f)}_{\bJ^\alpha_3(f)}.
        \end{align}
        \smallskip

        \begin{asparaitem}
            \item \underline{Estimates for $\bJ^\alpha_1(f)$.}
            We shall show that for $\alpha\geq 1/p-(n-1)/2$
            \begin{align}\label{eq:SS3.2case1est}
                \|\bJ^\alpha_1(f)\|_{L^p_{\rad}(\bH^n)}\lesssim T\,\e^{-c_pT}\|f\|_{L^{p,1}_{\rad}(\bH^n)},
            \end{align}
            and that for $\alpha>1/p-(n-1)/2$ the above estimate holds with $L^p_{\rad}$ in place of $L^{p,1}_{\rad}$.

            We first claim that for $R\geq 1/10$ and $\alpha\geq 1/p-(n-1)/2$,
            \begin{align}\label{eq:SS3.2case1claim}
                \sup_{T\leq t\leq T+1}\|K^\alpha_t(R,\cdot)\|_{L^{p',\infty}([\frac1{10},\infty),(\sinh r)^{n-1}\d r)}\lesssim T^{\frac1{p'}}\e^{-c_pT}(\sinh R)^{-\frac{n-1}p}
            \end{align}
            and that the same estimates holds with $L^{p'}$ in place of $L^{p',\infty}$ for $\alpha>1/p-(n-1)/2$. Once these estimates are prove, it follows from Lemma~\ref{lem:S2kerexp}\eqref{lem:S2kerexp(i)} and the Lorentz H\"older inequality (cf. \cite[the proof of (v) on p.~58]{Gra14GTM249}) that for $\alpha\geq 1/p-(n-1)/2$
            \begin{align*}
                {}&\bJ^\alpha_1(f)(R)\\
                \leq{}&\sup_{T\leq t\leq T+1}\|k^\alpha_t(R,\cdot)\|_{L^{p',\infty}([\frac1{10},\infty),(\sinh r)^{n-1}\d r)}\|f\|_{L^{p,1}([\max\{R-T-1,\frac1{10}\},R+T+1],(\sinh r)^{n-1}\d r)}\\
                \lesssim{}&T^{\frac1{p'}}\e^{-c_pT}(\sinh R)^{-\frac{n-1}p}\|f\|_{L^{p,1}([\max\{R-T-1,\frac1{10}\},R+T+1],(\sinh r)^{n-1}\d r)},
            \end{align*}
            and for $\alpha>1/p-(n-1)/2$
            \begin{align*}
                \bJ^\alpha_1(f)(R)\lesssim T^{\frac1{p'}}\e^{-c_pT}(\sinh R)^{-\frac{n-1}p}\|f\|_{L^p([\max\{R-T-1,\frac1{10}\},R+T+1],(\sinh r)^{n-1}\d r)}.
            \end{align*}
            Note that these estimates imply, for $j=0,1,\dots$,
            \begin{align*}
                {}&\|\bJ^\alpha_1(f)\|_{L^p((jT,(j+1)T],(\sinh R)^{n-1}\d R)}\\
                \lesssim{}&T\,\e^{-c_pT}\times
                \begin{cases}
                    \big\|f\big\|_{L^{p,1}([\max\{(j-1)T-1,\frac1{10}\},(j+2)T+1],(\sinh r)^{n-1}\d r)}, & \alpha\geq \frac1p-\frac{n-1}2,\\
                    \big\|f\big\|_{L^p([\max\{(j-1)T-1,\frac1{10}\},(j+2)T+1],(\sinh r)^{n-1}\d r)}, & \alpha>\frac1p-\frac{n-1}2.
                \end{cases}
            \end{align*}
            Arguing analogously to \eqref{eq:SS3.1strong} and \eqref{eq:SS3.1resweak} to sum over $j$, we obtain the desired estimate \eqref{eq:SS3.2case1est}.

            To verify the claimed estimates \eqref{eq:SS3.2case1claim} and its strong version, we apply Lemma~\ref{lem:S2kerest} and the identity
            \begin{align*}
                \cosh a-\cosh b=\frac{\e^a}2(1-\e^{-a+b})(1-\e^{-a-b})
            \end{align*}
            to obtain for $R,r\geq 1/10$ and $T\leq t\leq T+1$
            \begin{align*}
                |K^\alpha_t(R,r)|\lesssim{}&\underbrace{\1_{\{|R-r|<t\}}\e^{-\frac{n-1}2t}(\sinh R\,\sinh r)^{\frac{1-n}2}\1_{\{|t-|R-r||,|t-R-r|>\frac14\}}}_{K^\alpha_{t,1}(R,r)}\\
                {}&\ \ +\underbrace{\1_{\{|R-r|<t\}}\e^{-\frac{n-1}2t}(\sinh R\,\sinh r)^{\frac{1-n}2}\phi^\alpha(|t-|R-r||)\,\1_{\{|t-|R-r||\leq\frac14\}}}_{K^\alpha_{t,2}(R,r)}\\
                {}&\ \ +\underbrace{\1_{\{|R-r|<t\}}\e^{-\frac{n-1}2t}(\sinh R\,\sinh r)^{\frac{1-n}2}\psi^\alpha(|t-R-r|)\,\1_{\{|t-R-r|\leq\frac14\}}}_{K^\alpha_{t,3}(R,r)},
            \end{align*}
            where
            \begin{align*}
                \phi^\alpha(v)=
                \begin{cases}
                    v^{\alpha+\frac{n-3}2}, & \frac{1-n}2<\alpha<\frac{3-n}2,\\
                    1, & \alpha\geq \frac{3-n}2,
                \end{cases}
                \qquad\psi^\alpha(v)=
                \begin{cases}
                    v^{\alpha+\frac{n-3}2}, & \frac{1-n}2<\alpha<\frac{3-n}2,\\
                    \log \frac1v, & \alpha=\frac{3-n}2,\\
                    1, & \alpha>\frac{3-n}2.
                \end{cases}
            \end{align*}
            For $\alpha>(1-n)/2$,
            \begin{align*}
                \sup_{T\leq t\leq T+1}\|K^\alpha_{t,1}(R,\cdot)\|_{L^{p'}([\frac1{10},\infty),(\sinh r)^{n-1}\d r)}\lesssim T^{\frac1{p'}}\e^{-c_pT}(\sinh R)^{-\frac{n-1}p}.
            \end{align*}
            In addition, for $\alpha\geq 1/p-(n-1)/2$,
            \begin{align*}
                {}&\sup_{T\leq t\leq T+1}\|K^\alpha_{t,2}(R,\cdot)\|_{L^{p',\infty}([\frac1{10},\infty),(\sinh r)^{n-1}\d r)}+
                \sup_{T\leq t\leq T+1}\|K^\alpha_{t,3}(R,\cdot)\|_{L^{p',\infty}([\frac1{10},\infty),(\sinh r)^{n-1}\d r)}\\
                \lesssim{}&\e^{-c_pT}(\sinh R)^{-\frac{n-1}p},
            \end{align*}
            and the above estimate holds with $L^{p'}$ in place of $L^{p',\infty}$ for $\alpha>1/p-(n-1)/2$. Therefore, the desired estimates \eqref{eq:SS3.2case1est} and its strong version are proved.
            \smallskip

            \item \underline{Estimates for $\bJ^\alpha_2(f)$.}
            We shall show that for $\alpha<1/p$ such that \eqref{eq:SS3.1threeweak} holds,
            \begin{align}\label{eq:SS3.2case2est}
                \|\bJ^\alpha_2(f)\|_{L^{p,\infty}_{\rad}(\bH^n)}\lesssim \e^{-\frac{n-1}{p'}T}\|f\|_{L^{p,1}_{\rad}(\bH^n)},
            \end{align}
            and that for $p>2n/(n-1)$ and $\alpha=(1-n)/p$, the above estimate holds with $L^p_{\rad}$ in place of $L^{p,\infty}_{\rad}$ and $L^{p,1}_{\rad}$.

            In view of Lemma~\ref{lem:S2kerexp}\eqref{lem:S2kerexp(i)}, we further decompose
            \begin{align*}
                \bJ^\alpha_2(f)\leq \underbrace{\sup_{T\leq t\leq T+1}\1_{(t-\frac12,t+\frac12)}\big|M^\alpha_t\big(f\1_{(0,\frac1{10}]}\big)\big|}_{\bJ^\alpha_{2,1}(f)}+\underbrace{\sup_{T\leq t\leq T+1}\1_{(\frac1{10},t-\frac12]}\big|M^\alpha_t\big(f\1_{(0,\frac1{10}]}\big)\big|}_{\bJ^\alpha_{2,2}(f)}.
            \end{align*}
            For $|t-R|<1/2$ and $r\leq 1/10$, we have
            \begin{align*}
                \cosh t-\cosh(R-r)\sim (t-R+r)\,\e^T,
            \end{align*}
            \begin{align*}
                |\cosh(R+r)-\cosh t|\sim|R+r-t|\,\e^T,
            \end{align*}
            \begin{align*}
                \sinh R\,\sinh r\sim r\,\e^T,
            \end{align*}
            which implies
            \begin{align*}
                |K^\alpha(R,r)|\lesssim \e^{-(n-1)T}E^\alpha_{t-T+2}(R-T+2,r)
            \end{align*}
            according to Lemma~\ref{lem:S2kerest} and the definition of $E^\alpha_t$. Therefore, in view of Lemma~\ref{lem:SS3.1Eumaxest}, for $\alpha<1/p$ satisfying \eqref{eq:SS3.1threeweak}
            \begin{align*}
                \|\bJ^\alpha_{2,1}(f)\|_{L^{p,\infty}((T-\frac12,T+\frac32),(\sinh R)^{n-1}\d R)}\lesssim{}&\e^{-\frac{n-1}{p'}T}\big\|\fe^\alpha\big(f\1_{(0,\frac1{10}]}\big)\big\|_{L^{p,\infty}((\frac32,\frac72),R^{n-1}\d R)}\\
                \lesssim{}&\e^{-\frac{n-1}{p'}T}\|f\|_{L^{p,1}((0,\frac1{10}],(\sinh r)^{n-1}\d r)},
            \end{align*}
            and that for $p>2n/(n-1)$ and $\alpha=(1-n)/p$, the above estimate holds with $L^p_{\rad}$ in place of $L^{p,\infty}_{\rad}$ and $L^{p,1}_{\rad}$.

            In addition, for $t-R\geq 1/2$ and $r\leq 1/10$, it follows from Lemma~\ref{lem:S2kerest} that
            \begin{align*}
                |K^\alpha_t(R,r)|\lesssim \e^{-(n-1)T},
            \end{align*}
            which implies for all $1\leq p\leq\infty$ and $\alpha>(1-n)/2$
            \begin{align*}
                \|\bJ^\alpha_{2,2}(f)\|_{L^p((\frac1{10},T+\frac12),(\sinh R)^{n-1}\d R)}\lesssim \e^{-\frac{n-1}{p'}T}\|f\|_{L^p((0,\frac1{10}],(\sinh r)^{n-1}\d r)}.
            \end{align*}
            Therefore, the desired estimates \eqref{eq:SS3.2case2est} and its strong version are proved.
            \smallskip

            \item \underline{Estimates for $\bJ^\alpha_3(f)$.}
            We shall show that for $\alpha<1/p$ such that \eqref{eq:SS3.1threeweak} holds,
            \begin{align}\label{eq:SS3.2case3est}
                \|\bJ^\alpha_3(f)\|_{L^{p,\infty}_{\rad}(\bH^n)}\lesssim \e^{-\frac{n-1}pT}\|f\|_{L^{p,1}_{\rad}(\bH^n)},
            \end{align}
            and that for $p>2n/(n-1)$ and $\alpha=(1-n)/p$, the above estimate holds with $L^p_{\rad}$ in place of $L^{p,\infty}_{\rad}$ and $L^{p,1}_{\rad}$.

            In view of Lemma~\ref{lem:S2kerexp}\eqref{lem:S2kerexp(i)}, we further decompose
            \begin{align*}
                \bJ^\alpha_3(f)\leq \underbrace{\sup_{T\leq t\leq T+1}\1_{(0,\frac1{10}]}\big|M^\alpha_t\big(f\1_{(t-\frac12,t+\frac12)}\big)\big|}_{\bJ^\alpha_{3,1}(f)}+\underbrace{\sup_{T\leq t\leq T+1}\1_{(0,\frac1{10}]}\big|M^\alpha_t\big(f\1_{(0,t-\frac12]}\big)\big|}_{\bJ^\alpha_{3,2}(f)}.
            \end{align*}

            For $R\leq 1/10$ and $|t-r|<1/2$, we argue as in the estimate for $\bJ^\alpha_{2,1}(f)$ above to obtain
            \begin{align*}
                |K^\alpha_t(R,r)|\,(\sinh r)^{n-1}\lesssim E^\alpha_{t-T+2}(R,r-T+2).
            \end{align*}
            Therefore, in view of Lemma~\ref{lem:SS3.1Eumaxest}, for $\alpha<1/p$ satisfying \eqref{eq:SS3.1threeweak}
            \begin{align*}
                \|\bJ^\alpha_{3,1}(f)\|_{L^{p,\infty}((0,\frac1{10}],(\sinh R)^{n-1}\d R)}\lesssim{}&\big\|\fe^\alpha\big(f(\cdot+T-2)\1_{(\frac32,\frac72)}\big)\big\|_{L^{p,\infty}((0,\frac1{10}),R^{n-1}\d R)}\\
                \lesssim{}&\|f(\cdot+T-2)\|_{L^{p,1}((\frac32,\frac72),r^{n-1}\d r)}\\
                \sim{}&\e^{-\frac{n-1}pT}\|f\|_{L^{p,1}_{\rad}(\bH^n)},
            \end{align*}
            and that for $p>2n/(n-1)$ and $\alpha=(1-n)/p$, the above estimate holds with $L^p_{\rad}$ in place of $L^{p,\infty}_{\rad}$ and $L^{p,1}_{\rad}$.

            In addition, for $R\leq 1/10$ and $t-r>1/2$, it follows from Lemma~\ref{lem:S2kerest} that
            \begin{align*}
                |K^\alpha_t(R,r)|\lesssim \e^{-(n-1)T},
            \end{align*}
            which implies for all $1\leq p\leq\infty$ and $\alpha>(1-n)/2$
            \begin{align*}
                \|\bJ^\alpha_{3,2}(f)\|_{L^p((0,\frac1{10}),(\sinh R)^{n-1}\d R)}\lesssim \e^{-\frac{n-1}pT}\|f\|_{L^p_{\rad}(\bH^n)}.
            \end{align*}
            Therefore, the desired estimates \eqref{eq:SS3.2case3est} and its strong version are proved.
        \end{asparaitem}
        \smallskip

        Finally, we combine \eqref{eq:SS3.2decom} with estimates \eqref{eq:SS3.2case1est}--\eqref{eq:SS3.2case3est} and there strong versions to conclude the proof of Proposition~\ref{prop:SS3.2global}.
    \end{proof}

    \subsection{Conclusion of the proof of Theorem~\ref{thm:main}}

    \begin{proof}[Proof of Theorem~\ref{thm:main} (sufficiency part)]
        Note that
        \begin{align*}
            \fm^\alpha(f)\leq\fm^\alpha_0(f)+\sum_{T=1}^\infty\fm^\alpha_T(f).
        \end{align*}
        Since for $1<p<\infty$ and sufficiently small $\varepsilon>0$, $\sum_{T=1}^\infty(T+1)\,\e^{(\varepsilon-c_p)T}<\infty$, Propositions~\ref{prop:SS3.1local} and~\ref{prop:SS3.2global} imply the $L^{p,1}_{\rad}\to L^{p,\infty}_{\rad}$-boundedness of $\fm^\alpha$ whenever $\alpha<1/p$ and \eqref{eq:SS3.1threeweak} holds, as well as the $L^p_{\rad}$-boundedness of $\fm^\alpha$ when $p>2n/(n-1)$ and $\alpha=(1-n)/p$. Theorem~\ref{thm:main}\eqref{thm:main(ii)} is therefore proved.

        To show the sufficiency part of Thoerem~\ref{thm:main}\eqref{thm:main(i)}, we fix $1<p<\infty$,
        \begin{align}\label{eq:SS3.3strict}
            \alpha>
            \begin{cases}
                1-n+\frac np, & 1<p\leq 2,\\
                \frac1p-\frac{n-1}2, & 2<p\leq \frac{2n}{n-1},\\
                \frac{1-n}p, & p>\frac{2n}{n-1},
            \end{cases}
        \end{align}
        and choose $1<p_0<p<p_1<\infty$ sufficiently close to $p$ such that $\alpha$ satisfies \eqref{eq:SS3.3strict} as well as $\alpha<1/p$ at $p_0$ and $p_1$. Interpolation between the $L^{q,1}_{\rad}\to L^{q,\infty}_{\rad}$-boundedness of $\fm^\alpha$ at $q=p_0$ and $q=p_1$ (cf. \cite[Theorem 1.4.19]{Gra14GTM249}) yields the $L^p_{\rad}$-boundedness of $\fm^\alpha$. Thus all the desired $L^p_{\rad}$-boundedness of $\fm^\alpha$ for $1<p<\infty$ under the additional restriction $\alpha<1/p$ is proved. At $p=\infty$, it follows directly from \eqref{eq:S10order} that $\fm^0$ is bounded on $L^\infty(\bH^n)$. Finally, we apply Lemma~\ref{lem:S2raising} to remove the additional restriction $\alpha<1/p$ for $1<p<\infty$ and $\alpha=0$ at $p=\infty$. This proves the sufficiency part of Theorem~\ref{thm:main}\eqref{thm:main(i)}, and thus the whole sufficiency part of Theorem~\ref{thm:main} is concluded.
    \end{proof}

    \begin{proof}[Proof of Theorem~\ref{thm:main} (necessity part)]
        We shall construct three counterexapmles to show that the necessity conditions for the $L^p_{\rad}$-boundedness of $\fm^\alpha$ are $\alpha>1-n+n/p$, $\alpha>1/p-(n-1)/2$ and $\alpha\geq(1-n)/p$, respectively.
        \smallskip

        \begin{asparaitem}
            \item \underline{The condition $\alpha>1-n+n/p$.}
            Let $0<\delta<1/100$ be a fixed sufficiently small number to be determined later. For $0<\varepsilon<1/100$ we set
            \begin{align*}
                f_\varepsilon:=\1_{[\varepsilon(1+\delta),\varepsilon(1+2\delta)]}.
            \end{align*}
            For $5\varepsilon\leq R\leq 1/10$, we choose $t=R-\varepsilon$. If $r\in\supp f_\varepsilon$, then $|R-r|<t<R+r$,
            \begin{align*}
                \sinh R\,\sinh r\sim R\varepsilon,\qquad\cosh t-\cosh(R-r)&\sim R\delta\varepsilon.
            \end{align*}
            Since $F(a,b;c;0)=1$, by taking $\delta=\delta_{\alpha,n}$ sufficiently small such that the hypergeometric function term in Lemma~\ref{lem:S2kerexp}\eqref{lem:S2kerexp(ii)} is uniformly positive and in view of Lemma~\ref{lem:S2kerexp}\eqref{lem:S2kerexp(ii)}, we have
            \begin{align*}
                K^\alpha_t(R,r)\gtrsim R^{1-n-\alpha}\varepsilon^{\alpha-1},
            \end{align*}
            whence for $R$ as above
            \begin{align*}
                M_t^\alpha(f_\varepsilon)(R)\gtrsim R^{1-n-\alpha}\varepsilon^{\alpha-1}\int_{\varepsilon(1+\delta)}^{\varepsilon(1+2\delta)}r^{n-1}\d r\sim \left(\frac\varepsilon R\right)^{\alpha+n-1}.
            \end{align*}
            This, together with the fact that $\|f_\varepsilon\|_{L^p_{\rad}(\bH^n)}^p\sim\varepsilon^n$, implies
            \begin{align*}
                \|\fm^\alpha f_\varepsilon\|_{L^p_{\rad}(\bH^n)}^p\gtrsim{}&\|f_\varepsilon\|_{L^p_{\rad}(\bH^n)}^p\varepsilon^{p(\alpha+n-1)-n}\int_{5\varepsilon}^{\frac1{10}}R^{n-1-p(n-1+\alpha)}\d R\\
                \sim{}&\|f_\varepsilon\|_{L^p_{\rad}(\bH^n)}^p\times
                \begin{cases}
                    1, & \alpha>1-n+\frac np,\\
                    \log\frac1\varepsilon, & \alpha=1-n+\frac np,\\
                    \varepsilon^{p(n-1-\frac np+\alpha)}, & \alpha<1-n+\frac np
                \end{cases}.
            \end{align*}
            Thus $\alpha>1-n+n/p$ is necessary for the $L^p_{\rad}$-boundedness of $\fm^\alpha$.
            \smallskip

            \item \underline{The condition $\alpha>1/p-(n-1)/2$.}
            Let $\delta>0$ be a fixed sufficiently small number to be determined later. For $0<\varepsilon<\delta/2$ we set
            \begin{align*}
                f_\varepsilon(r):=(r-1)^{-\frac1p}\left(\log\frac1{r-1}\right)^{-1}\1_{(1+\varepsilon,1+\delta)}(r).
            \end{align*}
            If $r\in\supp f_\varepsilon$, $5/4\leq R\leq 3/2$ and $t=R-1$, then $|R-r|<t<R+r$,
            \begin{align*}
                \sinh R\sinh r\sim 1,\qquad\cosh t-\cosh(R-r)\sim r-1\leq\delta.
            \end{align*}
            Since $F(a,b;c;0)=1$, by taking $\delta=\delta_{\alpha,n}$ sufficiently small such that the hypergeometric function term in Lemma~\ref{lem:S2kerexp}\eqref{lem:S2kerexp(ii)} is uniformly positive and in view of Lemma~\ref{lem:S2kerexp}\eqref{lem:S2kerexp(ii)}, we have
            \begin{align*}
                K^\alpha_t(R,r)\gtrsim (r-1)^{\alpha+\frac{n-3}2},
            \end{align*}
            whence for $R$ as above and $\alpha\leq 1/p-(n-1)/2$
            \begin{align*}
                |M^\alpha_t(f_\varepsilon)(R)|\gtrsim \int_\varepsilon^\delta\frac{s^{\alpha-\frac1p+\frac{n-1}2-1}}{\log\frac1s}\d s\to\infty\qquad\text{as }\varepsilon\to 0^+.
            \end{align*}
            On the other hand, $\|f_\varepsilon\|_{L^p_{\rad}(\bH^n)}\lesssim 1$ uniformly in $\varepsilon$. Thus $\alpha>1/p-(n-1)/2$ is necessary for the $L^p_{\rad}$-boundedness of $\fm^\alpha$.
            \smallskip

            \item \underline{The condition $\alpha\geq (1-n)/p$.}
            Let $0<\delta<1/100$ be a fixed sufficiently small number to be determined later. For $0<\varepsilon<\delta/2$ we set
            \begin{align*}
                f_\varepsilon:=\1_{[1+\frac{\delta\varepsilon}2,1+\delta\varepsilon]}.
            \end{align*}
            For $\varepsilon\leq R\leq 2\varepsilon$, we choose $t=1-R+2\delta\varepsilon$. If $r\in\operatorname{supp} f_\varepsilon$, then $|R-r|<t<R+r$,
            \begin{align*}
                \sinh R\,\sinh r\sim \varepsilon,\qquad\cosh t-\cosh(R-r)&\sim \delta\varepsilon.
            \end{align*}
            Since $F(a,b;c;0)=1$, by taking $\delta=\delta_{\alpha,n}$ sufficiently small such that the hypergeometric function term in Lemma~\ref{lem:S2kerexp}\eqref{lem:S2kerexp(ii)} is uniformly positive and in view of Lemma~\ref{lem:S2kerexp}\eqref{lem:S2kerexp(ii)}, we have
            \begin{align*}
                K^\alpha_t(R,r)\gtrsim \varepsilon^{\alpha-1},
            \end{align*}
            whence for $R$ as above
            \begin{align*}
                M_t^\alpha(f_\varepsilon)(R)\gtrsim \varepsilon^\alpha.
            \end{align*}
            This, together with the fact that $\|f_\varepsilon\|_{L^p_{\rad}(\bH^n)}\sim\varepsilon^{1/p}$, implies
            \begin{align*}
                \|\fm^\alpha f_\varepsilon\|_{L^p_{\rad}(\bH^n)}\gtrsim\|f_\varepsilon\|_{L^p_{\rad}(\bH^n)}\varepsilon^{\alpha-1/p}\left(\int_\varepsilon^{2\varepsilon}R^{n-1}\d R\right)^{\frac1p}\sim \|f_\varepsilon\|_{L^p_{\rad}(\bH^n)}\varepsilon^{\alpha+\frac{n-1}p}.
            \end{align*}
            When $p=\infty$, the above estimates are understood by the convention $1/\infty=0$. Thus $\alpha\geq (1-n)/p$ is necessary for the $L^p_{\rad}$-boundedness of $\fm^\alpha$.\qedhere
        \end{asparaitem}
    \end{proof}

    \subsection*{Declaration of AI usage}
    During preparation of the manuscript, the author used GPT-6 Astra for assistance with concrete computations and language editing. The author takes responsibility for the mathematical arguments and the final content of the paper.

    \subsection*{Acknowledgements}
    The author would like to thank Naijia Liu and Professor Lixin Yan for helpful discussions.

    \providecommand{\bysame}{\leavevmode\hbox to3em{\hrulefill}\thinspace}
    \providecommand{\href}[2]{#2}

\end{document}